\documentclass[11pt]{amsart}
\usepackage{amssymb,latexsym,amsmath,amsthm,enumitem,mathrsfs,geometry}
\usepackage{fancyhdr}
\usepackage{color}
\usepackage{stmaryrd}
\usepackage{hyperref}
\hypersetup{
    colorlinks=true,
    linkcolor=red,
    filecolor=magenta,
    urlcolor=cyan,
    citecolor=green
    }
\usepackage{lineno}
\usepackage{diagbox}
\usepackage{array}
\usepackage{tikz}
\usetikzlibrary{arrows}
\usetikzlibrary{positioning}
\usepackage{float}
\usepackage{bbm}
\usepackage{cleveref}
\usepackage{dsfont}
\usepackage{graphicx}
\usepackage{soul}
\usepackage{verbatim}
\usepackage{comment}
\usepackage{caption} 
\usepackage[symbol*]{footmisc}

\newtheorem{theorem}{Theorem}

\newtheorem*{theorem*}{Theorem}
\theoremstyle{remark}
\newtheorem{remark}{Remark}

\theoremstyle{plain}
\newtheorem*{ES}{Essential Simplicity (ES)}
\newtheorem*{PCC}{Pair Correlation Conjecture (PCC)}

\definecolor{pink}{rgb}{1,.2,.6}
\definecolor{orange}{rgb}{0.7,0.3,0}
\definecolor{blue}{rgb}{.2,.6,.75}
\definecolor{green}{rgb}{.4,.7,.4}
\definecolor{purple}{RGB}{127,0,255}
\definecolor{myred}{RGB}{137,37,37}

\numberwithin{equation}{section}
\allowdisplaybreaks

\title[Zeta Zeros on the Critical Line ]{Zeta Zeros on the Critical Line }

\author[Goldston]{Daniel A. Goldston}
\address{Department of Mathematics and Statistics, San Jose State University}
\email{daniel.goldston@sjsu.edu}

\author[Suriajaya]{Ade Irma Suriajaya}
\address{Faculty of Mathematics, Kyushu University}
\email{adeirmasuriajaya@math.kyushu-u.ac.jp}

\keywords{Riemann zeta-function, zeros, pair correlation, simple zeros, critical zeros}
\subjclass[2010]{11M06, 11M26}

\begin{document}

\begin{abstract}
Montgomery in 1973 introduced the pair correlation method to study the vertical distribution of Riemann zeta-function zeros. This work assumed the Riemann Hypothesis (RH). One striking application was a short proof that at least 2/3 of zeta-zeros are simple zeros, the first result of its type. Over the last 50 years, most work on pair correlation of zeta-zeros has continued to assume RH. Here we show that if RH could be removed from Montgomery's simple zero proof, then this would also give a proof that 2/3 of the zeros are simple and on the critical line. This idea has been applied in several recent papers to obtain other results on the zeros. 



\end{abstract}
\date{\today}

\maketitle




\section{Introduction}

In 1973 Montgomery \cite{Montgomery73} proved, assuming the Riemann Hypothesis (RH), that at least 2/3 of the zeros of the Riemann zeta-function are simple. Prior to this, it was not known either unconditionally or on RH that there were even infinitely many simple zeros. Currently, by a refined version of Levinson's method, K. Pratt, N. Robles, A. Zaharescu, and D. Zeindler \cite{Pra20} in 2020 proved that at least $41.7\%$ of the zeros of $\zeta(s)$ are on the critical line, and at least $40.7\%$ of the zeros are on the critical line and simple. By large scale computation with careful error analysis, Platt and Trudgian \cite{PlaTru21} in 2021 showed that in the critical strip up to height 3,000,175,332,800, there are exactly 12,363,153,437,138 zeros and all of them are on the critical line and are simple. Conditionally, under the RH assumption one can obtain that at least $67.92\%$ of the zeros of $\zeta(s)$ are simple using Montgomery's pair correlation method \cite{CGL2020}, and with a different method \cite{CGG98,BuiHB} one obtains at least $70.37\%$ of the zeros are simple.

\section{Zeros of the Riemann zeta-function} 
\label{sec2}

Let $s=\sigma +it$ denote any point in the complex plane. The meromorphic function $\zeta(s)$ has a simple pole at $s=1$ with residue 1, and is analytic elsewhere. For $\sigma>1$, we have the Euler product formula $\zeta(s):=\sum_{n=1}^\infty n^{-s} = \prod_p(1- p^{s})^{-1}$, where the product runs over the primes $p$. Thus we see that $\zeta(s)\neq 0$ in the half-plane $\sigma >1$ since the individual terms in the Euler product are not zero and converge to $1$ as $p \to \infty$.
By the functional equation $\zeta(s) = \chi(s) \zeta(1-s)$, where $\chi(s) =2^s\pi^{s-1}\sin{\left(\frac{\pi s}{2}\right)}\Gamma(1-s)$, we now obtain the analytic continuation of $\zeta(s)$ for $\sigma <0$ where we see the only zeros in this half-plane are the {\it trivial zeros} $s = -2,-4,-6,-8, \ldots$ coming from the factor $\sin(\pi s/2)$ in $\chi(s)$. Thus all the complex zeros of $\zeta(s)$ lie within the strip $0\leq \sigma \leq 1$. Each of the two proofs of the prime number theorem in 1896 by Hadamard and by de la Vall{\'e}e Poussin gave different proofs that there are no zeros on the line $\sigma =1$, and therefore all the zeros must lie in the {\it critical strip} $0<\sigma <1$.

There is one further property of the Riemann zeta-function that is so easy to prove that its importance can often be overlooked; namely that there are no zeros on the real axis $s=\sigma$ for $0\le \sigma<1$. This follows from the formula 
\[(1 - 2^{1-s})\zeta(s) = \sum_{n=1}^{\infty} \frac{(-1)^{n+1}}{n^s}, \quad \text{for} \quad 0<\sigma <1, \]
since this shows that $(1-2^{1-\sigma})\zeta(\sigma) > 0$ so that $\zeta(\sigma) < 0$ on the real line for $0<\sigma <1$. We also see immediately from the functional equation that $\zeta(0)=-1/2$.

If an analytic function $F(s)$ is real on a segment of the real axis it will satisfy the reflection principle that $F(\bar{s})=\overline{F(s)}$ in the symmetric region on either side of that segment where $F$ is analytic. Denoting the complex zeros of the Riemann zeta-function by $\rho=\beta+i\gamma$, we see that
$\zeta(\rho) = 0 = \zeta(\bar{\rho})$ and the functional equation gives as well $\zeta(\rho) = 0 = \zeta(1-\rho)$. Thus,
all nontrivial zeros come in pairs symmetric with the real axis, therefore we have the zeros $\rho=\beta+i\gamma$ and $\bar{\rho}=\beta-i\gamma$, and 
zeros with $\beta\neq 1/2$ come in symmetric quadruples with respect to both the real axis and the {\it critical line} (namely, the vertical line $\sigma=1/2$):
$\rho =\beta + i\gamma ,\ \bar{\rho}=\beta-i\gamma, \ 1-\rho = 1-\beta -i\gamma, \ 1-\bar{\rho}= 1-\beta +i\gamma $. This elementary fact turns out to be fundamental later in this paper.

\section{Counting zeros with multi-sets}
Since the zeros are symmetric about the real axis, we only need to study the zeros above the real axis. Analytic function have their zeros counted by the argument principle, which counts zeros with their multiplicity. Thus a zero $\rho$ with multiplicity $m_\rho $ is counted as $m_\rho$ zeros. Therefore, rather than using \lq\lq sets" of zeros we will use in this paper multi-sets of zeros where a zero $\rho$ with multiplicity $m_\rho$ has $m_\rho$ copies of itself in the multiset. Thus we replace the set $ Z(T) = \{ \rho: \ 0<\gamma \le T \}$ of complex zeros $\rho$ up to height $T$ by the multiset 
\begin{equation} \label{calZ(T)} \mathcal{Z}(T) = \{ \rho : 0 < \gamma \le T,\ \rho \ \text{has} \ m_\rho\ \text{copies} \}. \end{equation}
The set $Z(T)$ is called the support of $\mathcal{Z}(T)$, written $\text{Supp}(\mathcal{Z}(T))$, and is the set of \emph{distinct} zeros.

 Let $ N(T)$ denote the number of zeros in the critical strip with $0< \gamma \le T$. In 1905 von Mangoldt proved that, for $T\ge3$,
\begin{equation} \label{N(T)sum} N(T)= | \mathcal{Z}(T)| = \sum_{\rho \in \mathcal{Z}(T)}1 = \frac{T}{2\pi}\log \frac{T}{2\pi} - \frac{T}{2\pi} + O(\log{T}) , \end{equation}
and we have the useful special cases
\begin{equation} \label{N(T)easy}N(T)~\sim~ \frac{T}{2\pi}\log{T} \quad \text{as} \ T\to \infty, \qquad \text{and} \qquad N(T+1)-N(T) = O(\log{T}).\end{equation}


\section{Montgomery's RH proof of simple zeros} 

In this section we assume RH so that $\rho= 1/2+i\gamma $, and prove the following result.
\begin{theorem}[Montgomery] \label{thm1} Assuming RH, at least 2/3 of the zeros of $\zeta(s)$ are simple. \end{theorem}
The proof is based on the following inequality
\begin{equation}\label{fejer0}
\sum_{\substack{\rho,\rho' \in \mathcal{Z}(T)\\ \gamma=\gamma'}}1 \le \sum_{\rho,\rho' \in \mathcal{Z}(T)} \left(\frac{\sin (\frac12 (\gamma-\gamma')\log{T})}{\frac12 (\gamma-\gamma')\log{T}}\right)^2 
= \left(\frac43+o(1)\right) \frac{T}{2\pi}\log{T},\qquad \text{as}
\quad T\to\infty.
\end{equation}
Here the first inequality is trivially true. The asymptotic formula for the Fej{\'e}r kernel sum over differences of zeros was proved by Montgomery assuming RH; the interested reader may see the proof in \cite{Montgomery73}. The Fej{\'e}r kernel sum itself is obtained without RH from the formula 
\begin{equation} \label{fejer1} \sum_{\rho,\rho' \in \mathcal{Z}(T)} \left(\frac{\sin (\frac12 (\gamma-\gamma')\log{T})}{\frac12 (\gamma-\gamma')\log{T}}\right)^2 = 2\int_0^1 \bigg|\sum_{\rho\in \mathcal{Z}(T)}T^{i\alpha \gamma} \bigg|^2 (1-\alpha) \, d\alpha ,\end{equation} 
which follows immediately from the elementary formula
\begin{equation}\label{fejer2} 2\, {\rm Re}\int_0^1e^{iw\alpha} (1-\alpha) \, d\alpha = \left(\frac{\sin(\frac{1}{2}w)}{\frac12 w }\right)^2.\end{equation}
Thus in \eqref{fejer0} the double sums are over two copies of the sequence of zeros with $0<\gamma \le T$, where each sequence already has $m_\gamma$ copies of each zero of multiplicity $m_\gamma$.

Now consider the first sum in \eqref{fejer0} with the condition that $\gamma=\gamma'$. If we think of the terms as ordered pairs $(\gamma, \gamma')$ where $\gamma=\gamma'$, then a simple zero gets counted once, but a double zero with $m_\gamma =2$ gets counted 4 times, since $\gamma_1=\gamma_2 = \gamma_1'=\gamma_2'$ and we get the 4 terms $(\gamma_1,\gamma_1')$, $(\gamma_1,\gamma_2')$, $(\gamma_2,\gamma_1')$, $(\gamma_2,\gamma_2')$. Thus a single zero of multiplicity $m_\gamma$ gets counted $m_\gamma^2$ times. The proof of Theorem \ref{thm1} is now just three lines.

\begin{proof}[Proof of Theorem \ref{thm1}]
Since by RH, $\gamma=\gamma'$ is true if and only if $\rho = \rho'$, we have
\begin{equation}\label{multiplicityweight}
\sum_{\substack{\rho,\rho' \in \mathcal{Z}(T)\\ \gamma=\gamma'}}1
= \sum_{\rho \in \mathcal{Z}(T)} m_\rho
\le \left(\frac43+o(1)\right) \frac{T}{2\pi}\log{T},\qquad \text{as}
\qquad T\to\infty. \end{equation}
Note the simple inequality
\begin{equation}\label{Mont_simple}
    \sum_{\substack{\rho \in \mathcal{Z}(T) \\ \rho \ \text{simple}}}1
    +
    \sum_{\rho \in \mathcal{Z}(T)} m_\rho
    =
    2 \sum_{\substack{\rho \in \mathcal{Z}(T) \\ m_\rho=1}} 1
    +
    \sum_{\substack{\rho \in \mathcal{Z}(T) \\ m_\rho\ge2}} m_\rho
    \ge 2 \sum_{\rho \in \mathcal{Z}(T)} 1 = 2N(T).
\end{equation}
Then using \eqref{N(T)sum} and \eqref{N(T)easy}, we have as $T\to\infty$,
\[\begin{split}
\sum_{\substack{\rho \in \mathcal{Z}(T) \\ \rho \ \text{simple}}}1
&\ge
2N(T)-\sum_{\rho \in \mathcal{Z}(T)} m_\rho
\ge \left(2-\frac43+o(1)\right)\frac{T}{2\pi}\log{T}
= \left(\frac23+o(1)\right)\sum_{\rho \in \mathcal{Z}(T)}1.
\end{split}\]
\end{proof}

\section{Montgomery's Simple Zero Method if RH is not assumed}

For 50 years it never seemed to occur to anyone (including the first-named author of this paper) to seriously think about this proof when RH is not assumed. As we will see in a moment, the key idea here is that, despite the title of this section, not assuming RH is not enough; what you have to do is take seriously the possibility that RH is false. The result of this mental block is we spent almost two years proving \eqref{fejer0} holds with conditions weaker than RH without seeing the connection to zeros on the critical line, which we call for brevity {\it critical zeros} in the rest of this exposition.

Recall for the first sum in \eqref{fejer0} and \eqref{multiplicityweight} we saw that, assuming RH, the condition $\gamma = \gamma'$ is equivalent to $\rho =\rho'$. But if RH is false then this is also sometimes false, since then there are in $\mathcal{Z}(T)$ at least two symmetric zeros $\rho = \beta +i\gamma$ and $1-\overline{\rho} = 1-\beta+i\gamma$ with $\beta \neq 1/2$.
Hence, if $\beta \neq 1/2$ then the pairs of zeros $(\beta+i\gamma, 1-\beta+i\gamma)$ and $(1-\beta+i\gamma, \beta+i\gamma)$ are also solutions of the equation $\gamma=\gamma'$ on the horizontal line $t=\gamma$. We call these \lq\lq symmetric diagonal terms". Thus, in addition to the diagonal terms $\rho=\rho'$ in \eqref{multiplicityweight}, which occur for every zero whether RH is assumed or not, we also need to include the symmetric diagonal terms for zeros with $\beta\neq 1/2$.
This is not the end. Since we are working with pairs of zeros, if there are three or more zeros on the same horizontal line, we can also have terms which are not diagonal and not symmetric.

To conclude, we have
\begin{align} \label{key}
\sum_{\substack{\rho,\rho' \in \mathcal{Z}(T)\\ \gamma=\gamma'}} 1 \ = \
\sum_{\rho \in \mathcal{Z}(T)}m_{\rho}
+
\sum_{\substack{\rho \in \mathcal{Z}(T)\\ \beta \neq \frac{1}{2}}}m_{\rho}
+
\sum_{\substack{\rho,\rho' \in \mathcal{Z}(T)\\ \ \beta+\beta' \neq 1,\, \gamma = \gamma'}} 1,
\end{align}
where the first sum on the right-hand side comes from the diagonal terms, the second is given by the symmetric diagonal terms, and the rest in the third sum are non-symmetric horizontal terms.

We can now immediately prove our RH-free version of Montgomery's simple zero result.

\begin{theorem} \label{thm2} Suppose there exists a constant $\mathbf{C}$ where $1\le\mathbf{C}<2$ and that, as $T\to \infty$, 
\begin{equation} \label{thm2input}
\sum_{\substack{\rho,\rho' \in \mathcal{Z}(T) \\ \gamma=\gamma'}} 1 \le \Big(\mathbf{C}+o(1)\Big)\frac{T}{2\pi}\log{T}.
\end{equation}
Then asymptotically at least the proportion $2-\mathbf{C}$ of the zeros of $\zeta(s)$ are simple, and at least the proportion $2-\mathbf{C}$ of the zeros of $\zeta(s)$ are on the critical line. 
\end{theorem}


We first remark on the restrictions on the constant $\mathbf{C}$ in our assumption.
Note that
\begin{equation}\label{lower_ineq}
\sum_{\substack{\rho,\rho' \in \mathcal{Z}(T)\\ \gamma=\gamma'}} 1
~\ge~
\sum_{\substack{\rho,\rho' \in \mathcal{Z}(T)\\ \rho =\rho'}} 1
~=~
\sum_{\rho \in \mathcal{Z}(T)} m_{\rho}
~\ge~
\sum_{\rho \in \mathcal{Z}(T)} 1
= N(T)
~\sim~
\frac{T}{2\pi}\log{T}
\end{equation}
as $T\to \infty$.
Thus $\mathbf{C}$ in our assumption has to satisfy $\mathbf{C}\ge1$. We also restrict $\mathbf{C}<2$, since for $\mathbf{C}\ge 2$ the method fails to give a positive proportion. 

\begin{proof}[Proof of Theorem \ref{thm2}] 

Substituting \eqref{key} into \eqref{thm2input}, we have
\begin{equation} \label{thm2proofeq} 
\sum_{\rho \in \mathcal{Z}(T)}m_{\rho}+ \sum_{\substack{\rho \in \mathcal{Z}(T)\\ \beta \neq \frac{1}{2}}}m_{\rho}+ \sum_{\substack{\rho,\rho' \in \mathcal{Z}(T)\\ \ \beta+\beta' \neq 1,\, \gamma = \gamma'}} 1\le \Big(\mathbf{C}+o(1)\Big) \frac{T}{2\pi}\log{T},\qquad \text{as}
\quad T\to\infty. \end{equation}
In this inequality, removing any sum on the left-hand side results in a smaller quantity since each is non-negative.
First, we have 
\[ \sum_{\rho \in \mathcal{Z}(T)} m_\rho \leq\Big(\mathbf{C}+o(1)\Big) \frac{T}{2\pi}\log{T},
\qquad \text{as} \quad T\to\infty, \]
which is \eqref{multiplicityweight} with $\mathbf{C}= 4/3$, and the identical argument \eqref{Mont_simple} used there gives
\begin{equation} \label{thm1proof1}\sum_{\substack{\rho \in \mathcal{Z}(T) \\ \rho\ \text{simple}}}1 \ge \Big(2-\mathbf{C}+o(1)\Big)\frac{T}{2\pi}\log{T},
\qquad \text{as} \quad T\to\infty, \end{equation}
which proves that the proportion of simple zeros of $\zeta(s)$ is $\ge 2-\mathbf{C}$.

Next, for the symmetric diagonal terms, we have, as $T\to\infty$,
\[ \sum_{\substack{\rho \in \mathcal{Z}(T)\\ \beta\neq 1/2}} m_\rho \le \Big(\mathbf{C}+o(1)\Big) \frac{T}{2\pi}\log{T} -\sum_{ \rho \in \mathcal{Z}(T)} m_\rho. \]
Recalling the lower bound in \eqref{lower_ineq}, we obtain
\begin{equation} \label{thm1proof2}\sum_{\substack{\rho \in \mathcal{Z}(T)\\ \beta\neq 1/2}} m_\rho \le \Big(\mathbf{C}-1+o(1)\Big) \frac{T}{2\pi}\log{T},
\qquad \text{as} \quad T\to\infty.
\end{equation}
Thus 
\[ \sum_{\substack{\rho \in \mathcal{Z}(T)\\ \beta= 1/2}}1= \sum_{\rho \in \mathcal{Z}(T)}1 - \sum_{\substack{\rho \in \mathcal{Z}(T)\\ \beta\neq 1/2}}1 \ge N(T)-\sum_{\substack{\rho \in \mathcal{Z}(T)\\ \beta\neq 1/2}} m_\rho \ge \Big(2-\mathbf{C}+o(1)\Big) \frac{T}{2\pi}\log{T}
\] 
as $T\to\infty$,
and we have proved that the proportion of zeros of $\zeta(s)$ on the critical line is $\ge 2-\mathbf{C}$.
\end{proof}

\section{Simple and critical zeros using diagonal and symmetric diagonal terms}

While we cannot improve either result in Theorem \ref{thm2}, we can obtain further results on zeros that are both simple and critical, simple or critical, and also the average of the proportion of simple zeros and the proportion of critical zeros.

\begin{theorem} \label{thm3} Suppose there exists a constant $ \mathbf{C}\ge 1$ and that, as $T\to \infty$, 
\begin{equation} \label{thm3input}
\sum_{\rho \in \mathcal{Z}(T)} m_{\rho} + \sum_{\substack{\rho \in \mathcal{Z}(T) \\ \beta \neq \frac{1}{2}}} m_{\rho}
\le \Big(\mathbf{C}+o(1)\Big) \frac{T}{2\pi}\log{T}. \end{equation}
Then asymptotically,
\begin{enumerate}[label={(\roman*)},itemsep=4pt]
\item \label{thm3-i} the proportion of zeros of $\zeta(s)$ which are simple and on the critical line is $\ge 2-\mathbf{C}$,
\item \label{thm3-ii} the average of the proportions of simple zeros and of critical zeros of $\zeta(s)$ is $\ge (3-\mathbf{C})/2$,
\item \label{thm3-iii} the proportion of zeros of $\zeta(s)$ which are either simple or critical or both is $\ge (4-\mathbf{C})/3$.
\end{enumerate}
\end{theorem}

\begin{remark} Notice that the proof of Theorem \ref{thm2} actually only uses the assumption \eqref{thm3input}, and \ref{thm3-i} of Theorem \ref{thm3} already implies Theorem \ref{thm2}. The diagonal and symmetric diagonal terms are easy to analyze in terms of the properties of \lq \lq simple" and \lq \lq critical", while we have dropped the non-symmetric horizontal terms which are more complicated. We return to this in the next section where we make use of \lq \lq horizontal multiplicity".
\end{remark}

\begin{remark} Note that if $\mathbf{C}=4/3$ as in Theorem \ref{thm1}, then by \ref{thm3-ii} of Theorem \ref{thm3} the average of the proportions of simple zeros and critical zeros is asymptotically $\ge 5/6$. Thus at least one of these proportions is $\ge 5/6$, although the method doesn't determine which. This agrees with the RH result in Theorem \ref{thm1} where the proportion of simple zeros is $\ge 2/3$ and proportion of critical zeros is 1, thus giving an average proportion $\ge 5/6$. We obtain in \ref{thm3-iii} of Theorem \ref{thm3} that $\ge 8/9$ of the zeros are either simple or critical or both. Equivalently, the proportion of zeros that are multiple zeros off the critical line is $\le 1/9$ . 
\end{remark}

Recall the multiset $\mathcal{Z(T)}$ from \eqref{calZ(T)}, and the set of distinct zeros $Z(T)= \{\rho: 0<\gamma\le T\} =\text{Supp}(\mathcal{Z}(T))$. We will consider $Z(T)$ the universe in what follows and take subsets $A\subset Z(T)$, and define the complement of $A'$ in $Z(T)$ to be
$A'= \{\rho: \rho\notin A,\ 0<\gamma\le T\}$. Similarly, we define the multiset $\mathcal{A}$ formed from $A$ by multiplicity and the complement $\mathcal{A'}$ in the universe $\mathcal{Z(T)}$ by
\begin{equation} \label{MultsetA} \mathcal{A} =\{ \rho \in A: 0 < \gamma \le T,\ \rho \ \text{has} \ m_\rho\ \text{copies} \}, \qquad \mathcal{A'} =\{ \rho \notin \mathcal{A}: 0 < \gamma \le T\}. \end{equation}
We will count the number of elements of $\mathcal{A}$ in two ways: $N(\mathcal{A})= |\mathcal{A}|$ where we count zeros with multiplicity, and $N^*(\mathcal{A})$ where we count zeros with multiplicity and also weighted by $m_\rho$. Thus
\begin{equation} \label{N,N*} N(\mathcal{A}) := \sum_{\substack{\rho \in \mathcal{A}\\ 0<\gamma \le T}}1, \qquad N^*(\mathcal{A}) := \sum_{\substack{\rho \in \mathcal{A}\\ 0<\gamma \le T}}m_\rho .\end{equation}
Note also that
\[ N(\mathcal{A'}) = N(\mathcal{Z}) - N(\mathcal{A})=N(T) - N(\mathcal{A}) .\]

We will now work with the sets of simple zeros $S$ and critical zeros $C$ defined by
\begin{equation} \label{S,C}
S := \{ \rho :\ 0<\gamma \le T,\ \rho\ \text{simple}\ (m_\rho=1)\}, \qquad C:= \{ \rho : \ 0<\gamma \le T, \ \rho \ \text{critical}\ (\beta=1/2)\}.\end{equation}
We thus have the multisets $\mathcal{S}$ and $\mathcal{C}$ as defined in \eqref{MultsetA}; note here $\mathcal{S}=S$.
With this preparation, we are now ready to prove Theorem \ref{thm3}.

\begin{proof}[Proof of Theorem \ref{thm3}]

With the usual Venn diagram for the two sets $S$ and $C$, we partition the zeros in $Z(T)$ into the usual four disjoint sets according to being or not being simple/critical, form the corresponding multisets, and count with multiplicity to obtain
\begin{equation} \label{countPartition}
N(\mathcal{Z})= N(T) = N(\mathcal{S}\cap\mathcal{C}) + N(\mathcal{S'}\cap\mathcal{C}) + N(\mathcal{S}\cap\mathcal{C'}) + N(\mathcal{S'}\cap\mathcal{C'}). \end{equation}
The same decomposition in \eqref{thm3input} gives on counting with multiplicity and weighting by multiplicity
\begin{equation} \label{count*} N^*(\mathcal{S}\cap\mathcal{C}) + N^*(\mathcal{S'}\cap\mathcal{C}) + 2N^*(\mathcal{S}\cap\mathcal{C'}) + 2N^*(\mathcal{S'}\cap\mathcal{C'})\le (\mathbf{C}+o(1))N(T).\end{equation}
Since $N^*(\mathcal{S}\cap\mathcal{C}) = N(\mathcal{S}\cap\mathcal{C})$ and $N^*(\mathcal{S}\cap\mathcal{C'}) = N(\mathcal{S}\cap\mathcal{C'})$, 
this equation becomes 
\[ N(\mathcal{S}\cap\mathcal{C}) + N^*(\mathcal{S'}\cap\mathcal{C}) + 2N(\mathcal{S}\cap\mathcal{C'}) + 2N^*(\mathcal{S'}\cap\mathcal{C'})\le (\mathbf{C}+o(1))N(T).\]
We eliminate the last two $N^*$'s in this equation using 
$N^*(\mathcal{S'}\cap\mathcal{C})\ge 2N(\mathcal{S'}\cap\mathcal{C})$ and $N^*(\mathcal{S'}\cap\mathcal{C'}) \ge 2N(\mathcal{S'}\cap\mathcal{C'})$, where these are equalities if $\mathcal{S'}$ only contains double zeros. Thus we obtain
\begin{equation} \label{finalreduction} N(\mathcal{S}\cap\mathcal{C}) + 2N(\mathcal{S'}\cap\mathcal{C}) + 2N(\mathcal{S}\cap\mathcal{C'}) + 4N(\mathcal{S'}\cap\mathcal{C'})\le (\mathbf{C}+o(1))N(T).\end{equation}
Now by \eqref{countPartition} we see 
\[ N(\mathcal{S'}\cap\mathcal{C}) + N(\mathcal{S}\cap\mathcal{C'}) + N(\mathcal{S'}\cap\mathcal{C'}) = N(T) - N(\mathcal{S}\cap\mathcal{C}),\]
and substituting this into \eqref{finalreduction}, we obtain 
\[N(\mathcal{S}\cap\mathcal{C}) + 2(N(T) - N(\mathcal{S}\cap\mathcal{C}))+ 2N(\mathcal{S'}\cap\mathcal{C'}) \le (\mathbf{C}+o(1))N(T). \]
Thus we obtain
\begin{equation} \label{mainfirstresult}N(\mathcal{S}\cap\mathcal{C}) \ge (2-\mathbf{C} +o(1))N(T) + 2N(\mathcal{S'}\cap\mathcal{C'}). \end{equation}
From this we immediately have $ N(\mathcal{S}\cap\mathcal{C}) \ge (2-\mathbf{C} +o(1))N(T)$, which proves \ref{thm3-i} of Theorem \ref{thm3}. 

To prove \ref{thm3-ii} of Theorem \ref{thm3}, we begin with the equalities
\begin{equation}\label{SunionC} N(\mathcal{S}\cup\mathcal{C}) = N(\mathcal{S}) +N(\mathcal{C}) - N(\mathcal{S}\cap\mathcal{C}) \quad \text{and} \quad N(\mathcal{S}\cup\mathcal{C}) = N(\mathcal{S}\cap\mathcal{C'}) +N(\mathcal{S}\cap\mathcal{C})+N(\mathcal{S'}\cap\mathcal{C}) ,\end{equation}
which on combining gives 
\[ N(\mathcal{S}) + N(\mathcal{C}) = 2N(\mathcal{S}\cap\mathcal{C})+N(\mathcal{S'}\cap\mathcal{C})+N(\mathcal{S}\cap\mathcal{C'}) .\]
On applying \eqref{countPartition}, this gives
\[ N(\mathcal{S}) + N(\mathcal{C}) = N(\mathcal{S}\cap\mathcal{C}) + N(T) - N(\mathcal{S'}\cap\mathcal{C'}), \]
and by \eqref{mainfirstresult} produces
\[ N(\mathcal{S}) + N(\mathcal{C}) \ge (3-\mathbf{C}+o(1)) N(T) + N(\mathcal{S'}\cap\mathcal{C'}) \ge (3-\mathbf{C}+o(1)) N(T).\]

To prove \ref{thm3-iii} of Theorem \ref{thm3}, first note that by \eqref{finalreduction} and the second equation in \eqref{SunionC} that 
\[ 2N(\mathcal{S}\cup\mathcal{C})-N(\mathcal{S}\cap\mathcal{C}) + 4N(\mathcal{S'}\cap\mathcal{C'})\le (\mathbf{C}+o(1))N(T).\]
Next, since $(\mathcal{S}\cup\mathcal{C})' = \mathcal{S'}\cap\mathcal{C'}$, we have $N( \mathcal{S'}\cap\mathcal{C'})= N(T)-N(\mathcal{S}\cup\mathcal{C})$, and substituting this gives 
\[ (4 -\mathbf{C}+o(1) )N(T) \le 2N(\mathcal{S}\cup\mathcal{C}) + N(\mathcal{S}\cap\mathcal{C})\le 3N(\mathcal{S}\cup\mathcal{C}).\]

\end{proof}

\section{Simple and critical zeros using horizontal multiplicity}

Soundararajan pointed out to us an alternative approach for obtaining the results in the last two sections. He also showed us \ref{thm3-iii} of Theorem \ref{thm3}.

In place of \eqref{thm2input} and \eqref{thm2proofeq}, we now use
\begin{equation} \label{Hmult1}
\sum_{\substack{\rho,\rho' \in \mathcal{Z}(T) \\ \gamma=\gamma'}} 1 = \sum_{\rho \in \mathcal{Z}(T)}H(\gamma) \le \Big(\mathbf{C}+o(1)\Big) \frac{T}{2\pi}\log{T},
\quad \text{where} \quad
H(\gamma) := \sum_{\substack{\rho' \in \mathcal{Z}(T)\\ \gamma'=\gamma}} 1.
\end{equation}
If we assume RH, then $H(\gamma) = m_\rho = m_\gamma$ the usual multiplicity of the zero $\rho$. Without RH, then $H(\gamma)$ counts with multiplicity the number of zeros on the horizontal line $t=\gamma$. Thus $H(\gamma)$ is a \lq \lq horizontal multiplicity" for the multiset of zeros on each horizontal line $t=\gamma$. Thus we see, for example:

\begin{itemize}
    \item If $H(\gamma)=1$, then on the line $t=\gamma$, we have one simple zero on the critical line.
    \item If $H(\gamma)=2$, then on the line $t=\gamma$, we have either a double zero on the critical line or a pair of two symmetric simple zeros. 
    \item If $H(\gamma)=3$, then on the line $t=\gamma$, we have either a triple zero on the critical line, or a simple zero on the critical line and a pair of two symmetric simple zeros.
\end{itemize}
Note that to have at least a double zero off the critical line on the line $t=\gamma$, $H(\gamma)$ needs to be at least $4$.

\begin{proof}[Proof of Theorem \ref{thm3}]
Note $\mathcal{Z}(T)$ is partitioned into disjoint multisets of zeros by 
\[ \mathcal{H}(k) := \{\rho \in \mathcal{Z}(T): H(\gamma)=k\}, \qquad k\in \mathbb{N}. \] 
We also see from above that 
\[
\mathcal{H}(1) \subset \mathcal{S}\cap\mathcal{C},\quad
\mathcal{H}(2) \subset (\mathcal{S'}\cap\mathcal{C})\cup (\mathcal{S}\cap\mathcal{C'}),\quad
\mathcal{H}(3) \subset (\mathcal{S}\cap \mathcal{C})\cup(\mathcal{S'}\cap \mathcal{C})\cup (\mathcal{S}\cap\mathcal{C'})
= \mathcal{S}\cup\mathcal{C}.
\]

We now assume \eqref{Hmult1}.
If $H(\gamma)=1$, using an argument analogous to \eqref{Mont_simple}, we have
\[ N(\mathcal{S}\cap\mathcal{C}) \ge \sum_{\substack{\rho \in \mathcal{Z}(T)\\ H(\gamma)=1}} 1 \ge \sum_{\rho \in \mathcal{Z}(T)} (2-H(\gamma))\ge (2-\mathbf{C}+o(1))N(T), \]
which proves \ref{thm3-i} of Theorem \ref{thm3}.
Next, 
\[ \begin{split} (3 -\mathbf{C}+o(1) )N(T)&\le \sum_{\rho \in \mathcal{Z}(T)} (3-H(\gamma))\le 2\sum_{\substack{\rho \in \mathcal{Z}(T)\\ H(\gamma)=1}} 1 +\sum_{\substack{\rho \in \mathcal{Z}(T)\\ H(\gamma)=2}} 1 \\ &
\le 2 N(\mathcal{S}\cap\mathcal{C}) + N(\mathcal{S}'\cap\mathcal{C}) + N(\mathcal{S}\cap\mathcal{C}') = N(\mathcal{S}) + N(\mathcal{C}), 
\end{split} \]
which proves \ref{thm3-ii} of Theorem \ref{thm3}. For \ref{thm3-iii} of Theorem \ref{thm3}, we have, since $\mathcal{H}(3)$ is disjoint from $\mathcal{H}(1)$ and $\mathcal{H}(2)$ and contained in $\mathcal{S}\cup\mathcal{C}$,  
\[ \begin{split}  (4-\mathbf{C}+o(1))N(T) &\le \sum_{\rho \in \mathcal{Z}(T)} (4-H(\gamma))\le 3\sum_{\substack{\rho \in \mathcal{Z}(T)\\ H(\gamma) =1 }} 1 +2\sum_{\substack{\rho \in \mathcal{Z}(T)\\ H(\gamma)=2}} 1 + \sum_{\substack{\rho \in \mathcal{Z}(T)\\ H(\gamma)=3}} 1\\ & \le 3N(\mathcal{S}\cap\mathcal{C}) + 2N(\mathcal{S}'\cap\mathcal{C}) + 2N(\mathcal{S}\cap\mathcal{C}') = 2N(\mathcal{S}\cup\mathcal{C}) +N(\mathcal{S}\cap\mathcal{C})
\\& \le
3N(\mathcal{S}\cup\mathcal{C}).
\end{split}\]
\end{proof}

\section{Two Results on Critical Zeros}

Using Theorem \ref{thm3}, we proved in \cite{BGST-CL} the following result.
\begin{theorem}[Baluyot, Goldston, Suriajaya, Turnage-Butterbaugh]\label{thm4} Assume that all of the zeros of $\zeta(s)$ with $T<\gamma \le 2T$ lie in the region
\begin{equation} B_b = \{ s=\sigma + it \ : \ \Big|\sigma - \frac12\Big|< \frac{b}{2\log{T}},\ T< t\le 2T \}.\end{equation} Then if $b=0.3185$, we have for the zeros of $\zeta(s)$ with $T<\gamma \le 2T$ that at least 2/3 are simple and on the critical line. 
\end{theorem}

We define the {\it average spacing between consecutive zeros} by 
\begin{equation} \label{averagespacing} \frac{\text{Length of}\ [0,T]}{N(T)} \sim \frac{T}{\frac{T}{2\pi}\log{T}} \sim \frac{2\pi}{\log{T}}, \quad \text{as} \ T\to \infty. \end{equation}
Thus in Theorem \ref{thm4} we are assuming all the zeros are within a vertical distance of $b/(4\pi)$ times the average spacing between consecutive zeros at height $T$. 

The following conjecture gives an asymptotic formula for the number of pairs of zeros counted with multiplicity within a distance $\lambda$ times the average spacing.

\begin{PCC}\label{PCC} For fixed $\lambda>0$ and $T\to\infty$,
\[
\sum_{\substack{\rho,\rho' \in \mathcal{Z}(T) \\ 0< \gamma-\gamma' \le \frac{2\pi\lambda}{\log{T}}}} 1 \sim
 \left(\int_0^\lambda\left(1-\left(\frac{\sin{\pi u}}{\pi u}\right)^2\right) du \right)\frac{T}{2\pi}\log{T}.
\]
\end{PCC}
\noindent Montgomery first made this conjecture in \cite{Montgomery73}. From this conjecture sprung the discovery that the zeros of $L$-functions obey a random matrix eigenvalue model; for $\zeta(s)$ this is the Gaussian Unitary Ensemble (GUE) model.

Since \hyperref[PCC]{{\bf PCC}} only concerns the non-zero vertical distance between zeros, it is unaffected by RH and nothing in \hyperref[PCC]{{\bf PCC}} by itself implies directly that any zeros are on the critical line.
Montgomery's reasoning to support \hyperref[PCC]{{\bf PCC}} used RH, but Gallagher and Mueller \cite{GaMu78} proved that \hyperref[PCC]{{\bf PCC}} implies asymptotically $100\%$ of the zeros are simple using a different method based on unconditional results of Selberg and Fujii. 
However, since RH was assumed and used in many parts of \cite{GaMu78}, we clarified this in \cite{GLSS1} to show the results on \hyperref[PCC]{{\bf PCC}} hold without RH.
As a consequence, by Theorem \ref{thm2}, we also obtain that asymptotically $100\%$ of the zeros are on the critical line.

\begin{theorem}[Goldston, Lee, Suriajaya, Schettler] \label{thm5}
The Pair Correlation Conjecture implies
asymptotically 100\% of the zeros of the $\zeta(s)$ are simple and on the critical line.
\end{theorem}

The proof of Theorem \ref{thm5} actually depends on the following property of \hyperref[PCC]{{\bf PCC}}.
\begin{ES} \label{ES}
If $\lambda \to 0$ as $ T\to \infty$, then we have 
\begin{equation} \label{ESeq} \sum_{\substack{\rho,\rho' \in \mathcal{Z}(T) \\ | \gamma-\gamma'| \le \frac{2\pi\lambda}{\log{T}}}} 1 = \Big(1+o(1)\Big)\frac{T}{2\pi} \log{T}. \end{equation}
\end{ES} 
Clearly \eqref{ESeq} implies $\mathbf{C}=1$, thus Theorems \ref{thm2} and \ref{thm3} follows from \hyperref[ES]{{\bf ES}}. This reduces the proof of Theorem \ref{thm5} to proving without RH that \hyperref[PCC]{{\bf PCC}} implies \hyperref[ES]{{\bf ES}}.

Other pair correlation conjectures besides \hyperref[PCC]{{\bf PCC}} also imply \hyperref[ES]{{\bf ES}}.
In \cite{GLSS2} we proved that the Alternative Hypothesis (AH) gives a different pair correlation density function from \hyperref[PCC]{{\bf PCC}}, which in one form implies \hyperref[ES]{{\bf ES}} and thus Theorem \ref{thm4} with \hyperref[PCC]{{\bf PCC}} replaced by that form of AH.

We mention that Selberg \cite{SelCollected2} and Bombieri and Hejhal \cite{BomHej95} have both made use of \hyperref[ES]{{\bf ES}} in their work on linear combinations of $L$-functions. In particular, the latter paper made use of Montgomery's Fej{\'e}r kernel method used in Theorem \ref{thm1} to obtain partial results without using \hyperref[ES]{{\bf ES}}.


\section*{Acknowledgment of Funding}

The authors thank the American Institute of Mathematics for its hospitality and providing a pleasant environment where our work on this project began. 
The second author is currently supported by the Inamori Research Grant 2024, JSPS KAKENHI Grant Number 22K13895, and Kyushu University International Research Leader Training Program (EBXU0101).


\bibliographystyle{alpha}
\bibliography{AHReferences}

\end{document}